\documentclass[11pt,reqno]{amsart}
\usepackage{tikz,a4wide}
\usetikzlibrary{arrows.meta}
\usetikzlibrary{math,calc}
\usepackage{float}
\usepackage{amsmath}
\usepackage{xfrac}
\usepackage{stmaryrd}
\SetSymbolFont{stmry}{bold}{U}{stmry}{m}{n}
\usepackage{mathrsfs,bm,amsthm,mathtools,yfonts,amssymb,color}
\usepackage{esint}
\usepackage{xcolor}
\usepackage[pagebackref,citecolor=black,linkcolor=blue,colorlinks=true]{hyperref}
\usepackage{tikz-3dplot}
\usepackage{epstopdf}
\usepackage{cleveref}
\usepackage{csquotes} 
\MakeOuterQuote{"} 

\makeatletter 
\DeclarePairedDelimiter{\@tmpabs}{\lvert}{\rvert}
\newcommand{\@absstar}[1]{{\@tmpabs*{#1}}}
\newcommand{\@absnostar}[2][]{{\@tmpabs[#1]{#2}}}
\newcommand{\abs}{\@ifstar\@absstar\@absnostar}
\makeatother

\makeatletter 
\DeclarePairedDelimiter{\@tmpnorm}{\lVert}{\rVert}
\newcommand{\@normstar}[1]{{\@tmpnorm*{#1}}}
\newcommand{\@normnostar}[2][]{{\@tmpnorm[#1]{#2}}}
\newcommand{\norm}{\@ifstar\@normstar\@normnostar}
\makeatother

\usepackage{cite}

\begingroup
\newtheorem{theorem}{Theorem}[section]
\newtheorem{lemma}[theorem]{Lemma}
\newtheorem{proposition}[theorem]{Proposition}

\endgroup

\theoremstyle{definition}
\newtheorem{definition}[theorem]{Definition}
\newtheorem{remark}[theorem]{Remark}

\newtheorem{conjecture}[theorem]{Conjecture}

\numberwithin{equation}{section}

\title[Thresholds for anomalous dissipation]
{Regularity thresholds for anomalous dissipation and related phenomena in passive scalars}

\author[M. Bagnara]{Marco Bagnara}
\address{Department of Mathematics, Imperial College London, London, SW7 2AZ, UK}
\email{m.bagnara@imperial.ac.uk}
\author[D.W. Boutros]{Daniel W. Boutros}
\address{Department of Applied Mathematics and Theoretical Physics, University of Cambridge, Cambridge CB3 0WA, UK}
\email{dwb42@cam.ac.uk}
\author[C. De Lellis]{Camillo De Lellis}
\address{School of Mathematics, Institute for Advanced Study, 1 Einstein Dr., Princeton NJ 08540, USA}
\address{Gran Sasso Science Institute, viale Francesco Crispi, 7, 67100 L’Aquila, Italy} 
\email{camillo.delellis@ias.edu}
\author[S. Mayboroda]{Svitlana Mayboroda}
\address{Department of Mathematics, ETH Z\"urich, R\"amistrasse 101, 8092 Z\"urich, Switzerland}
\address{School of Mathematics, Institute for Advanced Study, 1 Einstein Dr., Princeton NJ 08540, USA}
\email{svitlana.mayboroda@math.ethz.ch}

\date{July 16, 2026}



\allowdisplaybreaks

\begin{document}

\begin{abstract}
We prove the absence of anomalous dissipation for passive scalars driven by some random autonomous divergence-free vector fields in $\mathbb T^d$. In dimension $d=2$ we just need continuity almost surely and a mild nondegeneracy condition on the randomness. In dimension $d\geq 3$ we assume a special geometric structure and almost sure H\"older regularity with a H\"older exponent bigger than $\frac{1}{8}$. No regularity is assumed on the passive scalar except for boundedness in the initial data. The proof relies on dimension-theoretic arguments, as opposed to commutator estimates. A consequence of these results is that the same assumptions prevent (almost surely) many other expected properties of turbulent flows, such as anomalous regularization, the Yaglom-Obukhov-Corrsin law, and Richardson diffusion.
\end{abstract}

\maketitle

\tableofcontents

\section{Introduction}

The phenomenon of anomalous dissipation, which is the nonvanishing of the viscous energy dissipation in the vanishing viscosity limit, occupies a central position in the mathematical theory of turbulence. It is therefore sometimes referred to as the zeroth law of turbulence and it is expected to hold for (turbulent) solutions of the classical Navier-Stokes equations as the viscosity converges to zero (or equivalently as the Reynolds number becomes arbitrarily large). 
A closely related question arises for passive scalar transport. Consider a scalar $\theta_\varepsilon : \mathbb{T}^d \times [0,T] \rightarrow \mathbb{R}$ which solves the advection--diffusion equation
\begin{equation}\label{eq:advection_diffusion}
\partial_t \theta_\varepsilon + (v\cdot\nabla)\theta_\varepsilon
= \varepsilon \Delta \theta_\varepsilon,
\end{equation}
driven by a divergence-free autonomous velocity field $v : \mathbb{T}^d \times [0,T]\rightarrow \mathbb{R}^d$ ($\mathbb T^d$ denotes the $d$-dimensional flat torus). It is expected by physicists and confirmed by numerical studies that in a turbulent vector field the scalar variance dissipates at a non-vanishing rate as $\varepsilon \to 0$ \cite{ShraimanSiggia2000,DonzisSreenivasanYeung2005,Sreenivasan2019, Warhaft2000}. This motivates the following definition.

\begin{definition}\label{d:anomalous-dissipation}
Consider a bounded divergence-free vector field $v : \mathbb{T}^d \times [0,1] \rightarrow \mathbb{R}^d$. We say that $v$ allows a dissipation anomaly if there is a bounded initial data $\theta_{in} \in L^\infty (\mathbb{T}^d)$ with the following property. If $\theta_\varepsilon$ denotes, for any $\varepsilon>0$, the unique bounded solution of the advection-diffusion equation
\begin{align}\label{e:parabolic}
\left\{
\begin{array}{ll}
\partial_t \theta_\varepsilon + (v\cdot \nabla) \theta_\varepsilon = \varepsilon \Delta \theta_\varepsilon\\ \\
\theta_\varepsilon (\cdot, 0) = \theta_{in}
\end{array}
\right.\, ,
\end{align}
then the sequence $\{ \theta_\varepsilon \}$ satisfies
\[
\limsup_{\varepsilon \downarrow 0}\,  \varepsilon \int_0^T \int |\nabla \theta_\varepsilon|^2 (x,t)\, dx \, dt > 0\, .
\]
\end{definition}

Moreover, in the physics literature, beginning with the works of Obukhov, Corrsin, and Yaglom \cite{Obukhov1949,Corrsin1951,Yaglom1949} in 1949--1951, dimensional arguments analogous to the Kolmogorov's 1941 theory of turbulence predict nontrivial flux relations for passive scalars in turbulent regimes. In particular, heuristic cascade theories suggest a relation of the form $\alpha + 2\beta = 1$,
linking the H\"older regularity $C^\alpha (\mathbb{T}^d)$ of the velocity field to an effective regularity $C^\beta (\mathbb{T}^d)$ of the scalar. Such predictions are sometimes referred to as manifestations of anomalous regularization, whereby turbulent mixing produces effective smoothing properties that persist in the vanishing diffusivity limit. We will discuss this phenomenon in more detail in the Sections \ref{s:Yaglom} and \ref{s:Yaglom-2}. A third related phenomenon is that of Richardson's dispersion, roughly speaking the assertion that fluid particles get away from each other at a uniform rate independent of their initial position and of the viscosity, see Section \ref{s:no-richardson} for the precise definition.

\medskip

The present work focuses on the case of autonomous vector fields. Under the latter assumption and in two space dimensions we show that all of these phenomena are absent for random continuous vector fields. In $3$ or more space dimensions we are able to draw the same conclusion if we impose some special structure and assume H\"older regularity for the random field above the H\"older exponent $\frac{1}{8}$. The latter looks like a technical point and it seems likely that the special structure assumed in our theorems is an obstruction also at the level of random continuous fields in higher dimensions. 

In two space dimensions we consider probability measures on divergence-free autonomous periodic continuous vector fields satisfying a mild nondegeneracy condition, ensured for instance by the absolute continuity of the pointwise distribution.
Our main theorem shows that, under these assumptions, the associated transport equation possesses the DiPerna--Lions renormalization property almost surely. The latter property, whose definition dates back to the seminal work of DiPerna and Lions in \cite{DiPernaLions1989} and is stated precisely in Definition \ref{d:anomalous-dissipation}, seems to obstruct all the expected features of a turbulent flow. This is certainly clear to experts, but since it does not seem to be explicitly stated in a coherent way somewhere, we decided to collect precise statements and give precise proofs in the final part of this paper, without claiming any originality. 

Let us now rigorously state the main results.

\begin{theorem}[Two-dimensional case]\label{t:no-anomaly-2}
Consider a probability measure $\mathbb P$ on the space of continuous divergence-free autonomous vector fields $v$ on $\mathbb T^2$. Assume that 
\[
\mathbb P (\{v: v(x) =0\})=0 \qquad \mbox{for a.e. $x\in \mathbb T^2$.}
\] 
Then $\mathbb{P}$-almost surely $v$ satisfies the DiPerna-Lions renormalization property (cf. Definition \ref{d:renormalization}) and thus does not allow dissipation anomalies (cf. Theorem \ref{p:no-dissipation}), does not support Richardson dispersion (cf. Proposition \ref{p:no-Richardson}), does not display anomalous regularization (cf. Proposition \ref{p:no-regularization}), and violates Yaglom's law (under the interpretation of Theorem \ref{t:Yaglom-violated}). 
\end{theorem}

The main statement, namely that $v$ satisfies the DiPerna-Lions property, follows with a two lines proof: it suffices to apply Fubini's theorem to reduce it to a powerful deterministic result of Alberti, Bianchini, and Crippa, cf. \cite{AlbertiBianchiniCrippa2014}. In other words, it is rather trivial to see that, even under very mild assumptions, the stream function of a random continuous divergence-free vector field satisfies the ``weak Sard property'', cf. Definition \ref{d:weak-Sard}. The latter was introduced in \cite{AlbertiBianchiniCrippa2014} for the more general case of Lipschitz vector field, and the authors showed in fact that it is equivalent to the renormalization property and to the uniqueness of a regular Lagrangian flow in the sense of DiPerna and Lions.  In particular it is tempting to advance the following conjecture at the deterministic level:

\begin{conjecture}
An autonomous average-free and divergence-free vector field $v\in C^0 (\mathbb T^2, \mathbb R^2)$ allows dissipation anomalies if and only if its stream function violates the weak Sard property of Definition \ref{d:weak-Sard}.
\end{conjecture}

A good reason to believe in the above conjecture is given by the paper \cite{JohanssonSorella2024}. In the latter the authors, inspired by the works relating spontaneous stochasticity and anomalous dissipation (cf. \cite{Drivas}), construct explicit autonomous $2$-dimensional vector fields which allow dissipation anomalies, but the mechanism behind their construction is precisely the nonuniqueness of regular Lagrangian flows.

\medskip

Following the same blueprint, we can prove some extension of Theorem \ref{t:no-anomaly-2} to higher dimensions, although not quite as clean, see Theorem \ref{t:no-anomaly-general}. Here we will limit ourselves to the case of three dimensions because it is easier to state, physically more relevant, and moreover relates the structure of the random vector field to the classical ``Clebsch variables'' of the fluid-dynamics literature. 

\begin{theorem}[Three-dimensional case]\label{t:no-anomaly-3}
Consider a probability measure $\mathbb P$ on the space of continuous $C^1$ maps $\phi = (\phi_1, \phi_2): \mathbb T^3 \to \mathbb R^2$ such that:
\begin{itemize}
\item[(a)] There is $\alpha > \frac{1}{8}$ such that $\phi \in C^{1,\alpha} (\mathbb{T}^3 ; \mathbb{R}^2)$ $\mathbb{P}$-almost surely;
\item[(b)] There is a constant $C>0$ such that
\[
\mathbb P (\{\phi: |D\phi (x) - M_0|\leq r\}) \leq C r^6 \qquad \forall r>0, \quad \forall x\in \mathbb T^3\, ,\quad\mbox{and} \quad \forall M_0 \in \mathbb R^{2\times 3}\, .
\]
\end{itemize}
Then $\mathbb{P}$-almost surely the vector field $v = \nabla \phi_1 \times \nabla \phi_2$ satisfies the DiPerna-Lions renormalization property and hence all the conclusions of Theorem \ref{t:no-anomaly-2} apply. 
\end{theorem}

As already mentioned the proofs of Theorems \ref{t:no-anomaly-2} and \ref{t:no-anomaly-3} rely on a probabilistic Morse--Sard mechanism and a striking result of Alberti--Bianchini--Crippa \cite{AlbertiBianchiniCrippa2014} that shows that the renormalization property follows from a geometric condition. A statement of independent interest which we discovered during our investigations is that the classical Morse-Sard property for (suitable) random maps $\phi: \mathbb T^d \to \mathbb R^{d-j}$ holds, almost surely, with much less regularity than for deterministic maps.

\begin{theorem}\label{t:classical-Morse-Sard}
Let $d\geq 2$ and $1\leq j \leq d-1$. Consider a probability measure $\mathbb P$ on $C^0 (\mathbb T^d, \mathbb R^{d-j})$ and assume that 
\begin{itemize}
\item[(a)] $\phi \in C^{1,\alpha} (\mathbb{T}^d; \mathbb{R}^{d-j})$ almost surely for some $\alpha>\frac{j}{2+j}$;
\item[(b)] $\mathbb P (\{\phi : |D\phi (x_0) - M_0| \leq r \}) \leq C_0 r^{d(d-j)}$ for some fixed $C_0$ and all $x_0\in \mathbb T^d$ and $M_0 \in \mathbb R^{d(d-j)}$.
\end{itemize}
Then $\phi$ has almost surely the classical Morse-Sard property, namely $\phi^{-1} (y)$ is a $j$-dimensional $C^1$ submanifold for a.e. $y\in \mathbb R^{d-j}$. 
\end{theorem}

The lowest regularity is achieved for $j=1$, in which case the critical H\"older exponent is $\frac{1}{3}$. We do not know whether the exponent is sharp for any of the choices of $d$ and $j$.

\subsection{Context of our results} In a seminal 1949 note, Onsager \cite{Onsager1949} observed that weak solutions of the incompressible Euler equations conserve kinetic energy above H\"older regularity $1/3$, and may fail to conserve kinetic energy below H\"older regularity $1/3$. This prediction, now known as Onsager's conjecture, has since been resolved: energy conservation holds for $C^\alpha$ solutions with $\alpha>1/3$ \cite{Eyink1994, ConstantinETiti1994,DuchonRobert2000,CheskidovConstantinFriedlanderShvydkoy2008}, while non-conservative solutions exist for $\alpha<1/3$ \cite{DeLellisSzekelyhidi2009,DeLellisSzekelyhidi2010,DeLellisSzekelyhidi2013,BuckmasterDeLellisIsettSzekelyhidi2015, Isett2018, BuckmasterDeLellisSzekelyhidiVicol2019}. 
The exponent $1/3$ thus marks a sharp threshold separating rigidity from flexibility in nonlinear  incompressible fluid dynamics. The problem of providing a rigorous example of anomalous dissipation in the vanishing viscosity limit for the (unforced) Navier--Stokes equations in 3D remains open and it seems extremely challenging. 

For passive scalars there has been very important progress in recent years, even though the existing results still only scrape the surface of our understanding. In particular the literature is still far from providing a rigorous foundation for a phenomenological theory. Only a few deterministic examples of particular vector fields that give rise to anomalous dissipation and occasionally to some version of the Obukhov-Corrsin-Yaglom balance have been built, in \cite{DrivasElgindiIyerJeong2019,BrueDeLellis2023,ColomboCrippaSorella2023,JohanssonSorella2023,JohanssonSorella2024,BrueColomboCrippaDeLellisSorella2024,ArmstrongVicol2024,ElgindiLiss2024,BurczakSzekelyhidiWu2023} (see also \cite{HuysmansTiti2025}). At the level of random fields, the anomalous dissipation, anomalous regularization, and Richardson dispersion have been recently established for the first time in the groundbreaking work \cite{ArmstrongBouRabeeKuusi2026} for rather general classes of autonomous random vector fields with a {\it negative} H\"older exponent, close to $-1$. It is also proved that anomalous dissipation necessarily implies $\alpha + 2\beta \leq 1$, cf. \cite{DrivasElgindiIyerJeong2019,AkramovWiedemann2019}. Some of the aforementioned examples pertain to autonomous vector fields. For the sake of brevity, in this brief overview we will not discuss in detail the results in stochastic PDEs pertaining to the Kraichnan model and its relatives as the requirements on the roughness in time (white noise) on the vector field position them in a different context than the present note, but we refer for instance to \cite{BagnaraGrottoMaurelli2024,BernardGawedzkiKupiainen1998,DrivasGaleatiPappalettera2025,EyinkXin2000,GaleatiGrottoMaurelli2024,LeJanRaimond2004,Rowan2024,Rowan2025} (and see references therein).

Concerning the thresholds found in this paper, the arguments supporting the anomalous dissipation and the anomalous regularization for passive scalar transport, both in physics and in applied mathematics, rely on the dimensional analysis, energy cascade, and (sometimes) certain statistical homogeneities. To the best of our knowledge, none of them suggests the existence of a critical exponent in the range $(-1,1)$, in any dimension. In fact, one could argue that, since most of the existing deterministic examples, e.g., \cite{DrivasElgindiIyerJeong2019, ColomboCrippaSorella2023, JohanssonSorella2023} extend to $\alpha<1$, the literature so far suggested otherwise. The only exception is \cite{ArmstrongVicol2024}, which indeed restricts to $\alpha < 1/3$, but their example is also deterministic and it is hard to say whether it is a hard restriction or an artefact of the method. One should also reiterate that our vector fields are autonomous (but so are the ones in \cite{ArmstrongBouRabeeKuusi2026} and \cite{JohanssonSorella2023}) and the results are probabilistically a.s.. Finally, the work \cite{ArmstrongBouRabeeKuusi2026} was motivated by preexisting papers in the physics literature which use renormalization group ideas to analyze superdiffusivity and Richardson dispersion, cf. for instance \cite{BouchaudGeorges} (and also \cite{BouchaudComtetGeorgesLeDoussal1987,KonkonenKarjalainen1988}). Neither in this literature there seems to be an indication that some exponent $\alpha<1$ might be critical (see for instance the Richardson law in \cite[p. 216]{BouchaudGeorges}).

\subsection{Acknowledgments} The authors wish to thank Lucio Galeati for several useful comments to earlier versions of their paper. The research of the first author has been supported by the ERC/{\allowbreak}EPSRC Horizon Europe Guarantee EP/X020886/1 and the by the Istituto Nazionale di Alta Matematica (INdAM), group GNAMPA. The research of the second author has been supported by the Cambridge Trust and the Cantab Capital Institute for Mathematics of Information. The research of the third and fourth authors has been supported by the Simons Foundation through the Simons Initiative on the Geometry of Flows (Grant Award ID BD-Targeted-00017375 CDeL, SM). The fourth author is also partially supported by the Simons Collaboration on Localisation of Waves, 563916 SM.

\section{Reduction to the Morse-Sard property}\label{s:general}

\subsection{The DiPerna-Lions property} The following property was introduced in the seminal work of DiPerna and Lions \cite{DiPernaLions1989}, where the authors proved it for any vector field with Sobolev regularity $W^{1,1}$. It is a consequence of their theory that, when all bounded weak solutions of the transport equation are ``renormalized'', then uniqueness of the Cauchy problem and strong stability properties hold. 

\begin{definition}\label{d:renormalization}
Let $v$ be a bounded divergence-free vector field on $\mathbb T^d$ (not necessarily autonomous). A bounded weak solution $\theta$ of 
\begin{equation}\label{e:transport}
\left\{
\begin{array}{ll}
\partial_t \theta + (v\cdot \nabla )\theta = 0 \\ \\
\theta (\cdot, 0) = \theta_{in}
\end{array}
\right.
\end{equation}
is a bounded function $\theta$ which satisfies 
\begin{equation}\label{e:distributional}
\int_0^\infty \int_{\mathbb T^d} \theta (\partial_t \varphi + v\cdot \nabla \varphi)\, dx\, dt 
= \int_{\mathbb T^d} \theta_{in} (x) \varphi (x, 0)\, dx 
\end{equation}
for every test function $\varphi \in C^1_c (\mathbb T^d \times \mathbb R)$. 

The vector field $v$ is said to have the DiPerna-Lions renormalization property if, for every bounded initial data $\theta_{in}$, every bounded weak solution $\theta$ of \eqref{e:transport}, and every test function $\beta \in C^1 (\mathbb R)$, the scalar $\beta (\theta)$ is a weak solution of \eqref{e:transport} with initial data $\beta (\theta_{in})$.  
\end{definition}

An important extension of the renormalization property to $BV$ vector fields was later achieved by Ambrosio in \cite{Ambrosio2004}, while the novel Lagrangian approach in \cite{CrippaDeLellis2008} has spurred further extensions to other critical regularities, cf. \cite{BouchutCrippa}. The sharpness, at the deterministic level, of the threshold $\alpha=1$ in the Sobolev case is known since \cite{Depauw}.

\subsection{Dimension $d=2$ and the proof of (the renormalization part of) Theorem~\ref{t:no-anomaly-2}} 
In dimension two the DiPerna-Lions renormalization property for random autonomous \break divergence-free vector fields follows immediately from a simple Fubini argument and the main result of \cite{AlbertiBianchiniCrippa2014}. In fact, by the assumption of Theorem \ref{t:no-anomaly-2}, we have
\[
\int \mathcal{L}^2 (\{ x \in \mathbb{T}^2 : v (x) =0 \}) \, d{\mathbb P} (v) = \int_{\mathbb T^2} {\mathbb P} (\{v \in C^0 (\mathbb{T}^2; \mathbb{R}^2) : v(x) =0\}) \, dx = 0.
\]
and in particular with probability one the zero set of the vector field $v$ has measure zero. Assume now for the moment that the average of $v$ is zero, so that we can introduce the stream function $\phi$ of $v$. We then conclude that the measure 
\[
\mu := \mathbf{1}_{\{\nabla \phi =0\}} \mathcal{L}^2 
\]
vanishes identically $\mathbb{P}$-almost surely and in particular that $\phi$ satisfies the {\em weak Sard property} in \cite{AlbertiBianchiniCrippa2014}, which is the condition $\phi_\sharp \mu \perp \mathcal{L}^1$ (see Definition~\ref{d:weak-Sard}). This is then proved in \cite{AlbertiBianchiniCrippa2014} to be equivalent to the DiPerna-Lions renormalization property. As already mentioned, the derivation of the further conclusions in Theorem \ref{t:no-anomaly-2} will be given in Sections \ref{s:no-dissipation}-\ref{s:Yaglom-2}.

If $v$ does not have average $0$, then there is no global stream function for $v$ on the torus. However a stream function exists in every subset $\Omega$ which is simply connected. The result in \cite{AlbertiBianchiniCrippa2014} can then be used to show that, when $\Omega$ is simply connected and $\theta$ and $\beta$ are as in Definition \ref{d:renormalization}, then $\beta (\theta)$ will be a distributional solution of the transport equation \eqref{e:transport} in $\Omega \times [0, \infty)$ with initial data $\beta (\theta_{in})$. But a simple partition of unity allows then to show that $\beta (\theta)$ is in fact a distributional solution on the whole torus.

Observe that the argument for the weak Sard property are not really dimension dependent, namely the following holds.
\begin{definition}\label{d:weak-Sard}
Let $\phi\in C^1( \mathbb T^d, \mathbb R^{d-1})$ and let $Z$ be its critical set $Z=\{ x: {\rm rank}\, (D\phi (x)) \leq d-2\}$. Then we say that $\phi$ has the weak Sard property if $\phi_\sharp (\mathbf{1}_Z \mathcal{L}^d) \perp \mathcal{L}^{d-1}$.
\end{definition}

\begin{proposition}\label{p:weak-Sard}
Consider a probability measure on the space $C^1 (\mathbb T^d, \mathbb R^{d-1})$ such that 
$\mathbb P (\{\phi: {\rm rank}\, D\phi (x_0) < d-1\})=0$ for every $x_0$. Then its critical set $Z$
is, $\mathbb P$-almost surely, a Lebesgue null set and therefore $\phi$ has the weak Sard property. 
\end{proposition}

\subsection{Higher dimensions} Theorem \ref{t:no-anomaly-3} is a special case of the following more general one. In order to state it we first recall the Hodge $\star$ operator on (the flat torus) $\mathbb T^d$, which is a canonical way of mapping $k$-forms into $(d-k)$-forms. For our purposes we are interested only in the case $k=1$. If $x_1, \ldots , x_d$ are the canonical coordinates on $\mathbb T^d = \mathbb R^d/\mathbb Z^d$ the Hodge $\star$ maps the 1-form $\sum_{i=1}^d w_i dx_i$ to the $d-1$ form
\[
\star \sum_{i=1}^d w_i dx_i =  \sum_{i=1}^d (-1)^i w_i dx_1 \wedge \ldots \wedge dx_{i-1} \wedge \widehat{dx_i} \wedge dx_{i+1} \wedge \ldots \wedge dx_d\, .
\]
In particular the vector field $w= (w_1, \ldots, w_d)$ is divergence-free if and only if $d \star \sum_{i=1}^d w_i dx_i = 0$. 

Our main theorem is then the following.

\begin{theorem}\label{t:no-anomaly-general}
Consider a probability measure $\mathbb P$ on the space of continuous functions $\phi  = (\phi_1, \ldots, \phi_{d-1}) \in C (\mathbb T^d, \mathbb R^{d-1})$ satisfying the following two assumptions:
\begin{itemize}
\item[(a)] There is $\alpha > \frac{1}{8}$ such that $\phi \in C^{1,\alpha} (\mathbb{T}^d; \mathbb{R}^{d-1} )$ $\mathbb P$-almost surely;
\item[(b)] There is a constant $C$ such that 
\begin{equation}\label{e:upper-bound}
\mathbb P (\{\phi : |D\phi (x) - M_0|\leq r\}) \leq C r^{d(d-1)}\quad \forall r>0\, , \quad \forall x\in \mathbb T^d\, , \quad \mbox{and} \quad
\forall M_0 \in \mathbb R^{(d-1)\times d}\, .
\end{equation}
\end{itemize}
Then $\mathbb P$-almost surely, the divergence-free vector field $w$ such that 
\begin{equation}\label{e:Hodge}
\star \sum_{i=1}^d  w_i dx_i = d\phi_1 \wedge \ldots \wedge d\phi_{d-1}
\end{equation}
has the DiPerna-Lions renormalization property.
\end{theorem}

The key to the proof of Theorem \ref{t:no-anomaly-general} is again to use the insight of Alberti, Bianchini, and Cripppa in \cite{AlbertiBianchiniCrippa2014}. First of all, note that the components $\phi_k$ of the vector map $\phi$ satisfy $\nabla \phi_k \cdot w \equiv 0$, namely one of the conditions required in \cite[Section 6.2]{AlbertiBianchiniCrippa2014}. This orthogonality relation is easy to verify from the properties of the Hodge $\star$ operator. In fact for every vector field $z$ we have the relation
\[
(z\cdot w) dx_1 \wedge \ldots \wedge dx_d = \star \sum_{i=1}^d  w_i dx_i \wedge \sum_i z_i dx_i = d\phi_1 \wedge \ldots \wedge d\phi_{d-1} \wedge \sum_i z_i dx_i\, .
\]
In particular we conclude that, for any $k = 1, \ldots, d - 1$,
\[
(\nabla \phi_k \cdot w) dx_1 \wedge \ldots \wedge dx_d = d\phi_1 \wedge \ldots \wedge d\phi_{d-1}\wedge d\phi_k = 0\, .
\]
Next, by Proposition \ref{p:weak-Sard} $\phi$ satisfies the weak Sard property almost surely. In this case, however, the weak Sard property alone is not enough to guarantee the renormalization property. In fact the authors in \cite{AlbertiBianchiniCrippa2014} need additionally the condition that ``triods'' are absent for a.e. level set of $\phi$ (for the precise definition of a triod we refer to \cite[Section 2.3]{ABC-2}) and in the subsequent paper \cite{ABC-2} they show that this is not a consequence of the weak Sard property and might fail for $C^{1,\alpha}$ maps $\phi$. On the other hand, we can prove that this property does hold under the assumptions of Theorem \ref{t:no-anomaly-general}. 

\begin{proposition}\label{p:no-triods}
Let $\mathbb P$ be a probability measure satisfying the same assumptions as in Theorem \ref{t:no-anomaly-general}. Then $\mathbb P$-almost surely the map $\phi \in C^{1,\alpha} (\mathbb{T}^d ; \mathbb{R}^{d-1})$ satisfies the following property:
\begin{itemize}
\item[(NT)] For a.e. $y\in \mathbb R^{d-1}$ the level set $\phi^{-1} (\{y\})$ contains no triods.
\end{itemize}
\end{proposition}

With Proposition \ref{p:no-triods} at hand, we can use the results of \cite{AlbertiBianchiniCrippa2014} to conclude the renormalization property for any such $\phi$ which satisfies the weak Morse-Sard property, cf. Section 6.2 and Section 6.3 in \cite{AlbertiBianchiniCrippa2014}. 

\subsection{Absence of triods}\label{s:dimension-estimate}

The proof of Proposition \ref{p:no-triods} will be split into two steps. In the first step we will take advantage of  probabilistic arguments. The statement is as follows.

\begin{lemma}\label{l:probabilistic} Let $d\geq 2$, $1\leq j \leq d-1$ and $k\in \{0, \ldots , d-1-j\}$. 
Consider a probability measure $\mathbb P$ on $C^{1,\alpha} (\mathbb{T}^d; \mathbb{R}^{d-j})$ which satisfies condition (b) in Theorem \ref{t:classical-Morse-Sard}, for $\alpha \in (0,1)$. Then the Hausdorff dimension of the set 
\begin{equation}\label{e:critical}
Y := \{ x \in \mathbb{T}^d : {\rm rank}\, (D\phi (x)) \leq k\}
\end{equation}
is, $\mathbb P$-almost surely, at most $d-(d-k)(d-j-k)\alpha$.
\end{lemma}

Note that $(d-k)(d-j-k)$ is the codimension, in $\mathbb R^{(d-j)\times d}$ of the set of matrices with rank no larger than $k$.

We then use a purely deterministic argument in the second part of the proof.

\begin{lemma}\label{l:deterministic} Let $d$, $j$, and $k$ be as in Lemma \ref{l:probabilistic}.
Let $\phi \in C^{1,\alpha} (\mathbb T^d, \mathbb R^{d-j})$ be a function such that the Hausdorff dimension of 
the set $Y$ in \eqref{e:critical} is at most $\beta$. Then the set $\phi (Y)$ has Hausdorff dimension at most $\frac{\beta +\alpha k}{1+\alpha}$. 
\end{lemma}

Before proving the lemma we show how Proposition \ref{p:no-triods} and Theorem \ref{t:classical-Morse-Sard} both follow from it.

\begin{proof}[Proof of Theorem \ref{t:classical-Morse-Sard}]
Let $\mathbb P$, $j$, and $\alpha$ be as in the statement and 
choose $k=d-j-1$. In this case $Y$ in \eqref{e:critical} is the set $Z$ of critical points of $\phi$, while $\phi (Z)$ is the set of critical values. But Lemma \ref{l:probabilistic} implies that the dimension of $Z$ is, almost surely, no larger than $d-(j+1)\alpha$, while Lemma \ref{l:deterministic} implies (by taking $\beta = d - (j+1) \alpha$) that $\phi (Z)$ has Hausdorff dimension no larger than $\frac{d- (j+1)\alpha + (d-j-1)\alpha}{1+\alpha}$. If the latter number is $<d-j$, then $\phi (Z)$ is a (Lebesgue) null set of $\mathbb R^{d-j}$. Therefore we have 
\[
\frac{d-(j+1) \alpha + (d-j-1)\alpha}{1+\alpha} < d-j
\]
if and only if $\alpha > \frac{j}{j+2}$.
\end{proof}

\begin{proof}[Proof of Proposition \ref{p:no-triods}]
Let $\mathbb P$ and $\alpha$ be as in Proposition \ref{p:no-triods}. Consider $j=1$, $k= d-3$ and let $Y$ be as in \eqref{e:critical}. Note that $Y$ is closed and so its complement $U$ is open. According to \cite[Proof of Lemma 2.16]{ABC-2} the set 
\[
\{y\in \mathbb R^{d-1} : \phi^{-1} (y)\cap U \mbox{ contains a triod}\}
\]
is a (Lebesgue) null set in $\mathbb R^{d-1}$. Therefore it suffices to show that 
\begin{itemize}
\item[(A)] $\mathbb P$-almost surely the set $\phi (Y)$ is a null set.
\end{itemize}
First observe that Lemma \ref{l:probabilistic} gives that, $\mathbb P$-almost surely, the dimension of $Y$ is at most $d-6\alpha$. Hence Lemma \ref{l:deterministic} gives that, $\mathbb P$-almost surely, the dimension of $\phi (Y)$ is at most $\frac{d-6\alpha + (d-3)\alpha}{\alpha+1}$. As in the previous proof (i.e. of Theorem \ref{t:classical-Morse-Sard}) we just need the inequality
\[
\frac{d-6\alpha + (d-3)\alpha}{\alpha+1} < d-1
\]
to conclude (A). The above inequality is equivalent to $\alpha > \frac{1}{8}$, thus completing the proof.
\end{proof}

\subsection{Proof of the probabilistic dimension bound (Lemma \ref{l:probabilistic})} In all the arguments, we always consider functions $\phi$ in the regularity class $C^{1,\alpha} (\mathbb{T}^d; \mathbb{R}^{d-j})$.

\medskip

{\bf Step 1.} We consider the Euclidean space $\mathbb R^{(d-j)\times d}$ of $(d-j)\times d$ real matrices and denote by $W_k$ the closed subset of matrices of rank at most $k$. Denote by $c_0 = (d-k)(d-j-k)$ its codimension. 
If we denote by ${\rm dim}_H (E)$ the Hausdorff dimension of a set $E$ and introduce the set 
\[
\mathcal{Y} := \{\phi : {\rm dim}_H ((D\phi)^{-1} (W_k))> d-c_0\alpha\}\, ,
\]
then we can rephrase the claim of the lemma as $\mathbb P (\mathcal{Y})=0$, where we regard $D \phi$ as a map from $\mathbb{T}^d$ to $\mathbb{R}^{(d-j) \times d}$. On the other hand $\mathcal{Y}$ is the countable union of
\begin{equation}\label{e:pieces}
\mathcal{Y}_N := \{\phi : \|\phi\|_{C^{1,\alpha}} \leq N 
\quad \mbox{and}\quad \mathcal{H}^{d-c_0 \alpha + 1/N} ((D\phi)^{-1} (W_k\cap B_N)) > 0\}\, ,
\end{equation}
where $B_N\subset \mathbb R^{(d-j)\times d}$ denotes the closed ball of radius $N\in \mathbb N$ (centered at the origin) and $\mathcal{H}^\beta$ is the $\beta$-dimensional Hausdorff measure. In particular it suffices to show that
\begin{equation}\label{e:bounded-claim}
\mathbb P (\mathcal{Y}_N)=0 \qquad \forall N\in \mathbb N\setminus \{0\}\, .
\end{equation}

\medskip

{\bf Step 2.} From now on we fix a positive natural number $N$ and aim at proving \eqref{e:bounded-claim}. First of all we recall that $W_k$ is a real-algebraic subset of $\mathbb R^{(d-j)\times d}$ of codimension $c_0$. Its intersection with the closed ball $B_N$ has, therefore, finite Minkowski $(d (d-j)-c_0)$-dimensional content. In other words there is a constant $C$ (depending on both $W_k$ and $N$ but not on $\delta$) such that, for every $\delta>0$, 
\[
|\{ A \in \mathbb{R}^{(d-j)\times d} : {\rm dist}\, (A, W_k \cap B_N)  \leq \delta\}|\leq C \delta^{c_0}\, .
\]
Consider now the cubical decomposition of $\mathbb R^{d(d-j)}$ in cubes $Q:= \delta M + [0, \delta)^{d(d-j)}$, for $M$ ranging in the lattice $\mathbb Z^{(d-j)\times d}$. The cubes are chosen so to be pairwise disjoint. We then let $\mathcal{Q}_\delta$ be the subcollection of such cubes which intersect the set $\{ A \in \mathbb{R}^{(d-j)\times d} :{\rm dist}\, (A, W_k\cap B_N) \leq \delta\}$. The cardinality of this collection is controlled by $C \delta^{-d(d-j)+c_0}$, where $C$ is another constant independent of $\delta$: this follows because the cubes are pairwise disjoint and their union is contained in $\{ A \in \mathbb{R}^{(d-j)\times d} :{\rm dist}\, (A, W_k \cap B_N) \leq \delta (1+\sqrt{d})\}$.

Observe next that, by assumption (b) of Theorem \ref{t:classical-Morse-Sard}, for any $\delta>0$ and any fixed $Q$ in $\mathcal{Q}_\delta$ (for $x \in \mathbb{T}^d$)
\[
\mathbb{P} (\{\phi: D\phi (x)\in Q\}) \leq C \delta^{d(d-j)} \, . 
\]
In particular, summing over all the cubes $Q$ in $\mathcal{Q}_\delta$, we conclude immediately the bound
\begin{equation}\label{e:first-probability-bound}
\mathbb{P} (\{\phi: {\rm dist}\, (D\phi (x), W_k\cap B_N) \leq \delta\}) \leq C \delta^{c_0}\, .
\end{equation}

\medskip

{\bf Step 3.} We next fix a natural number $m$, we subdivide $\mathbb T^d$ in $2^{m d}$ closed nonoverlapping cubes $Q$ of sidelength $2^{- m}$, and for each $\phi\in C^{1,\alpha}$ we count the number $n( m,\phi)$ of cubes in the subdivision which intersect $(D\phi)^{-1} ( W_k \cap B_N)$. Let us enumerate the cubes of the subdivision as $Q_i$ and note that
\begin{equation}\label{e:hits}
n( m,\phi) = \sum_{i=1}^{2^{m d}} \mathbf{1}_{E_i} (\phi)
\end{equation}
where $E_i$ is the set of $\phi\in C^{1,\alpha} (\mathbb{T}^d ; \mathbb{R}^{d-j})$ for which there is a point $y\in Q_i$ such that $D\phi (y)\in W_k\cap B_N$.

Note that, if $x_i$ is the center of the cube $Q_i$ and $\|D\phi\|_{C^{1,\alpha}}\leq N$, then 
\begin{equation} \label{e:holder}
|D\phi (y) - D\phi (x_i)|\leq N |x-y_i|^\alpha 
\leq N (\sqrt{d}\, 2^{- m-1})^\alpha\, .
\end{equation}
Thus, if we introduce the set
\[
F_i := \{\phi \in C^{1,\alpha} (\mathbb{T}^d; \mathbb{R}^{d-j}) : {\rm dist}\, (D\phi (x_i), W_k\cap B_N) \leq N (\sqrt{d}\, 2^{-m-1})^\alpha\}\, ,
\]
we infer from estimate \eqref{e:holder} that $E_i\cap \{|D\phi\|_{C^{1,\alpha}}\leq N\} \subset F_i$ and 
\begin{equation}\label{e:hits-2}
n (m, \phi) \mathbf{1}_{\{\|\phi\|_{C^{1,\alpha}}\leq N\}} (\phi) \leq \sum_{i=1}^{2^{m d}} \mathbf{1}_{F_i} (\phi)\, .
\end{equation}
By inequality \eqref{e:first-probability-bound}, by choosing $\delta = N (\sqrt{d}\, 2^{- m -1})^\alpha$ we have $\mathbb{P} (F_i) \leq C N^{c_0} 2^{-c_0 m \alpha}$ and therefore, integrating \eqref{e:hits-2} and by taking into account that there are $2^{m d}$ distinct cubes in the cubical decomposition of $\mathbb T^d$ considered above, we conclude immediately
\begin{equation}\label{e:expectation}
\int_{\|D\phi\|_{C^{1,\alpha}}\leq N} n(m, \phi)\, d\mathbb{P} (\phi) \leq \sum_{i=1}^{2^{m d}} \mathbb{P} (F_i) \leq C N^{c_0} 2^{(d-c_0 \alpha) m}\, .
\end{equation}
We next use Chebyshev's inequality for the function $n (m, \phi) \mathbf{1}_{\{\|\phi\|_{C^{1,\alpha}}\leq N\}}$ to estimate the probability of the event 
\[
\mathcal{Y}_{N, m} := \left\{\phi: \|\phi\|_{C^{1,\alpha}}\leq N \quad\mbox{and} \quad n (m, \phi) \geq  2^{(d-c_0\alpha+1/(2N)) m} \right\}\, 
\]
and get therefore
\begin{equation}\label{e:Borel-Cantelli}
\mathbb P (\mathcal{Y}_{N,m}) \leq 2^{-(d-c_0\alpha+1/(2N))m}\,  \mathbb{E} \big[ n (m, \phi) \mathbf{1}_{\{\|\phi\|_{C^{1,\alpha}}\leq N\}} (\phi) \big]\leq C N^{c_0} 2^{- m/(2N)}\, .
\end{equation}
By the Borel-Cantelli lemma, the set 
\[
\mathcal{Y}'_N := \bigcap_{j=1}^\infty \bigcup_{m \geq j} \mathcal{Y}_{N,m}
\]
is a set of probability zero. 

We now claim that $\mathcal{Y}_N\subset \mathcal{Y}'_N$, which would conclude the proof. In fact, assume $\phi \in \mathcal{Y}_N$. If $\phi$ is in the complement of $\mathcal{Y}'_N$ then necessarily there would be a sequence of natural numbers $m (i)\to\infty$ with the property that $\phi \not\in \mathcal{Y}_{N, m (i)}$ for every $i$. But, given the definition of $\mathcal{Y}_N$, $\|\phi\|_{C^{1,\alpha}}\leq N$. This means that, for every $i$, $n ( m (i), \phi) < 2^{(d-c_0\alpha+1/(2N)) m (i)}$. Therefore we can cover $(D\phi)^{-1} ( W_k \cap B_N)$ with less then $2^{(d-c_0\alpha+1/(2N)) m (i)}$ cubes of sidelength $2^{- m (i)}$. We then conclude, from the definition of the $\mathcal{H}^\beta_\infty$ premeasure, that 
\begin{equation}\label{e:final-bound}
\mathcal{H}^{d-c_0\alpha + N^{-1}}_\infty ((D\phi)^{-1} (W_k \cap B_N)) 
\leq \omega_{d-c_0\alpha+N^{-1}} 2^{- m (i)/(2N)} \left(\frac{\sqrt{d}}{2}\right)^{d-c_0\alpha+N^{-1}}\, .
\end{equation}
But letting $i\uparrow \infty$ we would then conclude $\mathcal{H}^{d-c_0\alpha+N^{-1}}_\infty ((D\phi)^{-1} (W_k \cap B_N)) =0 $. Since $\mathcal{H}^\beta_\infty$-null sets and $\mathcal{H}^\beta$-null sets coincide, we then would be in contradiction with the definition of $\mathcal{Y}_N$, which requires $\mathcal{H}^{d-c_0 \alpha+N^{-1}} ((D\phi)^{-1} (W_k \cap B_N))>0$.  Therefore we have $\mathcal{Y}_N\subset \mathcal{Y}'_N$, from which it follows that both $\mathbb{P} (\mathcal{Y}_N)$ and $\mathbb{P} (\mathcal{Y})$ are equal to zero.

\subsection{Proof of Lemma \ref{l:deterministic}} 
From the definition of Hausdorff measure it follows that for every positive $\varepsilon \leq 1$ and $\delta>0$ there is a cover of $Y$ with sets $U_{l}$ such that 
\begin{equation} \label{e:upper-bound-1}
\sum_{l} ({\rm diam}\, (U_{l}))^{\beta+\delta} < \varepsilon\, .
\end{equation}
We then discard any $U_{l}$ which does not intersect $Y$ and for the remaining ones we pick a point $x_i\in U_{l} \cap Y$. Set $r_i := {\rm diam}\, (U_i)$ and conclude that the closed balls $B_{r_i} (x_i)$ cover $Y$ and satisfy (by rewriting equation \eqref{e:upper-bound-1})
\begin{equation}\label{e:upper-bound-2}
\sum_l r_l^{\beta+\delta} < \varepsilon\, .
\end{equation}
Next we recall that $\phi$ is Lipschitz and thus $\phi (U_i) \subset B_{C_0 r_i} (\phi (x_i))$, where $C_0 = {\rm Lip}\, (\phi)$. Moreover, the rank of the linear map $L_i := D\phi (x_i)$ is at most $k$, as $x_i \in Y$: this allows us to choose a $k$-dimensional vector space $V_i$ with the property that ${\rm Im}\, (L_i) \subset V_i$. 

Select then any unit normal vector $\nu_i$ which is orthogonal to $V_i$ and observe that, by Lagrange's mean-value theorem, for every $y\in B_{r_i} (x_i)$ there is a point $\xi$ on the segment $[x_i, y]$ such that  
\[
\nu_i \cdot (\phi (y)-\phi (x_i)) = \nu_i \cdot (D\phi (\xi) \cdot (y-x_i))\, .
\]
Since $\nu_i \cdot (D\phi (x_i) \cdot (y-x_i)) = 0$ (due to the orthogonality of $\nu_i$ with respect to $V_i$), we then conclude that 
\begin{align*}
|\nu_i \cdot (\phi (y) - \phi (x_i))| &= 
|\nu_i \cdot ((D\phi (\xi)-D\phi (x_i))\cdot (y-x_i))|
\leq [D \phi]_{0, \alpha} |\xi - x_i|^\alpha |y-x_i|\\
& \leq [D \phi]_{0, \alpha} r_i^{1+\alpha}\, .
\end{align*}
In particular the orthogonal projection of $\phi (B_{r_i} (x_i))$ onto the orthogonal complement $V_i^\perp$ of $V_i$ is contained in a closed $(d-j-k)$-dimensional ball $\sigma_i$ of diameter bounded $2 [D \phi]_{0, \alpha} r_i^{1+\alpha}$ centered on the projection of $\phi (x_i)$ onto $V_i^\perp$. If we denote by $z_i$ the projection of $\phi (x_i)$ onto $V_i$, the arguments above prove the inclusion
\[
\phi (B_{r_i} (x_i)) \subset (B_{C_0 r_i} (z_i) \cap V_i) + \sigma_i =: W_i\, .
\]
On the other hand, since $B_{C_0 r_i} (z_i)\cap V_i$ is a closed $k$-dimensional disk, the set $W_i$ can be covered by $N_i \leq \bar C r_i^{-\alpha k}$ balls  (of $\mathbb R^{d-1}$) with radius $r_i^{1+\alpha}$, where $\bar C$ depends only on $C_0$, $[D \phi]_{0, \alpha}$, and $d$. The union over $i$ of all these balls covers therefore $\phi (Y)$ and, if we fix $\gamma>0$, we can write
\begin{align}
\mathcal{H}^\gamma_\infty (\phi (Y)) &\leq C \sum_i N_i (r_i^{1+\alpha})^\gamma \leq \bar C \sum_i r_i^{(1+\alpha)\gamma - \alpha k}\label{e:Hausdorff-bound}
\end{align}
Choose now $\gamma = \frac{\beta+\alpha k + \delta}{1+\alpha}$ so that 
\[
(1+\alpha) \gamma -\alpha k = \beta + \delta
\]
In particular we have, by inequality \eqref{e:upper-bound-2},
\[
\mathcal{H}^\gamma_\infty (\phi (Y)) \leq \bar C \varepsilon\, ,
\]
where the constant $\bar C$ depends on $\phi$, $\beta$, and $\delta$ but not on $\varepsilon$. In particular letting $\varepsilon \downarrow 0$ we conclude that 
\[
\mathcal{H}^{\frac{\beta+\delta+k\alpha}{1+\alpha}} (\phi (Y))=0
\]
On the other hand, given that $\delta>0$ is arbitrary, this proves that the Hausdorff dimension of $\phi (Y)$ is no larger than $\frac{\beta+k\alpha}{1+\alpha}$. 

\section{Consequences of the DiPerna-Lions theory}

As mentioned in the introduction, we collect here a number of consequences of the DiPerna-Lions theory. The focus is to show that a vector field which satisties the DiPerna-Lions renormalization property necessarily violates many of the expected ``laws'' of turbulent fluids. This seems to be folklore in the literature, but we believe it is useful to give precise statements and include the corresponding arguments. 

\subsection{From the renormalization property to the absence of dissipation anomalies}\label{s:no-dissipation}
We provide for the reader's convenience a quick derivation of the following proposition from statements which can be found in the literature on the transport equation. 

\begin{theorem}\label{p:no-dissipation}
Consider a bounded divergence-free vector field $v: \mathbb T^d \times [0,T] \to \mathbb R^d$ which satisfies the renormalization property . Let $\theta_{in}$ be a bounded initial datum and for any $\varepsilon>0$ we denote by $\theta_\varepsilon$ the unique bounded solution of equation \eqref{e:parabolic}. Then we have
\begin{equation}\label{e:no-dissipation}
\lim_{\varepsilon \downarrow 0}\, \varepsilon \int_0^T \int_{\mathbb T^d} |\nabla \theta_\varepsilon|^2 (x,t)\, dx\, dt = 0
\end{equation}
and $\theta_\varepsilon$ converges strongly in $C([0,T], L^2 (\mathbb{T}^d))$ to the unique bounded weak solution of \eqref{e:transport}.
\end{theorem}
\begin{proof}
First of all the existence and uniqueness of the bounded distributional solution $\theta_0$ of \eqref{e:transport} is the main result of the DiPerna-Lions theory and the reader can consult e.g. \cite[Theorem 3.3 \& Proposition 3.6]{delellis}. The strong continuity of $t\mapsto \theta (\cdot, t)$ (up to a suitable choice of representative) is another well-known fact and the reader can consult e.g. \cite[Remark 25]{ambrosio}. By the maximum principle $\theta_\varepsilon$ is uniformly bounded in $L^\infty ((0,T); L^\infty (\mathbb{T}^d))$ and so, up to subsequence, it converges weakly-$*$ to some bounded function $\theta$. We immediately see (passing to the limit in the weak formulation \eqref{e:distributional}) that $\theta$ is a bounded distributional solution of equation \eqref{e:transport} and the aforementioned uniqueness implies that $\theta = \theta_0$. Fix a smooth test function $\varphi\in C^\infty (\mathbb T^d)$ and define $\Phi_\varepsilon (t) := \int_{\mathbb{T}^d} \varphi (x) \theta_\varepsilon (x,t) dx$. Multiplying equation \eqref{e:parabolic} with the function $\varphi (x)$ and integrating in space, we immediately obtain that 
\[
\frac{d}{dt} \Phi_\varepsilon (t) = \int_{\mathbb{T}^d} \theta_\varepsilon (x,t) \left[v (x,t) \cdot \nabla \varphi (x) + \varepsilon \Delta \varphi (x)\right]\, dx\, .
\]
In particular the functions $t\mapsto \Phi_\varepsilon (t)$ are equi-Lipschitz. By the Ascoli-Arzel\`a Theorem and the uniqueness (and time continuity) of the limit $\theta_0$, we conclude that 
\[
\lim_{\varepsilon \downarrow 0} \Phi_\varepsilon (t) = \int_{\mathbb{T}^d} \theta_0 (x,t) \varphi (x)\, dx\, .
\]
The arbitrariness of the test function $\varphi$ and the boundedness of $\theta_\varepsilon (\cdot, t)$ implies therefore that $\theta_\varepsilon (\cdot, t) \rightharpoonup \theta_0 (\cdot, t)$ in $L^2 (\mathbb T^d)$ for every $t \in [0,T]$. In particular, recalling that 
\begin{equation}\label{e:energy-identity}
\int_{\mathbb{T}^d} \theta^2_\varepsilon (x,T)\, dx = \int_{\mathbb{T}^d} \theta_{in}^2 (x)\, dx 
- 2\varepsilon \int_0^T \int_{\mathbb{T}^d} |\nabla \theta_\varepsilon|^2 (x,t)\, dx\, dt 
\end{equation}
and using the lower semicontinuity of the $L^2$ norm under weak convergence, we conclude
\begin{equation}\label{e:chiave}
\int_{\mathbb{T}^d} \theta_0^2 (x,T) \, dx \leq \liminf_{\varepsilon \downarrow 0} \int_{\mathbb{T}^d} \theta_\varepsilon^2 (x,T)\, dx 
\leq \int_{\mathbb{T}^d} \theta_{in}^2 (x)\, dx\, .
\end{equation}
However, the renormalization property gives that $\theta_0^2$ is the unique bounded weak solution of \eqref{e:transport} with initial data $\theta_{in}^2$ and testing the equation against functions of the type $\varphi (x,t) = \varphi (t)$ we conclude that $\|\theta_0 (\cdot, t)\|_{L^2}^2$ is constant. So the inequalities in \eqref{e:chiave} are actually equalities and in fact $\theta_\varepsilon (\cdot, T)$ converges strongly in $L^2 (\mathbb{T}^d)$ to $\theta (\cdot, T)$. Coupling this information with \eqref{e:energy-identity} we conclude \eqref{e:no-dissipation}. Note also that the argument immediately implies that $\theta_\varepsilon (\cdot, t)$ converges strongly to $\theta_0 (\cdot, t)$ in $L^2 (\mathbb{T}^d)$ for every $t \in [0,T]$. To upgrade pointwise convergence to uniform convergence we observe that we already know the uniform convergence in $L^2 (\mathbb{T}^d)$ endowed with the weak topology and that the functions $t\mapsto \|\theta_\varepsilon (\cdot, t)\|_{L^2}$ are converging to $t\mapsto \|\theta_0 (\cdot, t)\|_{L^2} \equiv \|\theta_{in}\|_{L^2}$ uniformly.
\end{proof}

\subsection{Anomalous regularization}\label{s:Yaglom} A widely accepted theory in the turbulence literature is that there is, in some statistical sense, a relation between the regularity of the advecting vector field $v$ and the passive scalar $\theta_\varepsilon$ in \eqref{e:parabolic}. Roughly speaking a $C^\alpha$ regularity for $v$ should correspond to a uniform $C^\beta$ regularity of $\theta_\varepsilon$ for $\beta$ satisfying $\alpha + 2\beta =1$. Some authors in the literature call this ``anomalous regularization''. More generally the ``law'' is known as Obukhov-Corrsin relation or the Yaglom law. We will discuss the Yaglom law in Section \ref{s:Yaglom-2}. Here we will show that anomalous regularization does not occur, if the random vector field $v$ satisfies the assumptions of Theorem \ref{t:no-anomaly-2} and Theorem \ref{t:no-anomaly-general}. 

\begin{proposition}\label{p:no-regularization}
Let $v$ be an autonomous divergence-free vector field which satisfies the conditions of Theorems \ref{t:no-anomaly-2} or \ref{t:no-anomaly-general}. Then there exist initial data $\theta_{in} \in L^\infty (\mathbb{T}^d)$ such that the corresponding unique solution $\theta$ to the transport equation \eqref{e:transport} has the property that $\theta (\cdot, t) \notin C^\beta (\mathbb{T}^d)$ for any $\beta > 0$ and $t \in [0,T]$. Moreover, for any sequence of solutions $\{ \theta_\epsilon \}$ to the advection-diffusion equation \eqref{eq:advection_diffusion} such that $\epsilon \rightarrow 0$, one has for any $\beta > 0$ and $t \in [0,T]$
\begin{equation}
\lim_{\epsilon \rightarrow 0} \lVert \theta_\epsilon (\cdot, t) \rVert_{C^\beta (\mathbb{T}^d)} = \infty.
\end{equation}
\end{proposition}
\begin{proof}
Choose $\theta_{in}$ to be a bounded function which takes a discrete set of values (more than one). Then the solution $\theta$ to the transport equation \eqref{e:transport} is constant along the trajectories of the unique regular Lagrangian flow $\Phi$ of the vector field $v$. Therefore the distribution function of $a\mapsto |\{\theta > a\}|$ will remain exactly the same. Therefore $\theta (\cdot, t)$ will be forced to remain discontinuous on a set of positive $\mathcal{H}^1$ measure. This immediately obstructs any regularization mechanism and thus prevents the tracer $\theta$ from becoming Hölder continuous.

In order to prove the second part, it was shown in Theorem \ref{p:no-dissipation}, that when the conclusion of Theorem \ref{t:no-anomaly-2} or \ref{t:no-anomaly-general} applies and the initial data $\theta_{in}$ are bounded, the solutions $\theta_\varepsilon$ of \eqref{e:parabolic} converge (strongly in $C ([0,T], L^2 (\mathbb{T}^d))$) to the corresponding unique 
solution $\theta$ of \eqref{e:transport}. If $\lVert \theta_\epsilon (\cdot, t) \rVert_{C^\beta (\mathbb{T}^d)}$ would be bounded uniformly in $\epsilon$ for some positive $\beta$ and $t \in [0,T]$, it would mean by an interpolation argument that $\theta (\cdot, t)$ would be in $C^{\beta-} (\mathbb{T}^d)$ which contradicts the previous result.
\end{proof}

\begin{remark}
As pointed out to us by Lucio Galeati, it is also obvious that, for a vector field as in Proposition \ref{p:no-regularization}, for any fixed positive time $t$ and any space of functions $X$ which embeds in $L^\infty (\mathbb{T}^d)$ but it is strictly contained in it, we can find an $L^\infty (\mathbb{T}^d)$ initial data $\theta_{in}$ with the property that the unique solution $\theta$ of the transport equation satisfies $\theta (\cdot, t)\not \in X$ and, consequently, $\|\theta_\varepsilon (\cdot, t)\|_X \to \infty$ as $\varepsilon\downarrow 0$. It suffices in fact to fix an element $\theta_{ter} \in L^\infty (\mathbb{T}^d)\setminus X$, solve the transport equation backward in time with data $\theta (\cdot, t) = \theta_{ter}$ and set $\theta_{in} = \theta (\cdot, 0)$. Note however that the conclusion of Proposition \ref{p:no-regularization} is stronger in another respect, namely that $\theta (\cdot, t')\not \in C^\beta (\mathbb{T}^d)$ for {\em every time} $t' \in [0,t]$.  
\end{remark}

\subsection{Richardson dispersion}\label{s:no-richardson} In this section we show why the DiPerna-Lions renormalization property obstructs the occurrence of Richardson dispersion. First, for every $\varepsilon$ we consider the stochastic ODE
\begin{equation}\label{e:sODE}
\left\{
\begin{array}{l}
dX_t^\varepsilon (x) = v (X_t^\varepsilon (x)) dt + \sqrt{2\varepsilon} dB_t\\ \\
X_0^\varepsilon (x) =x 
\end{array}\right.
\end{equation}
where $B_t$ is the standard $d$-dimensional Brownian motion. If $v$ is a turbulent vector field, Richardson dispersion is the prediction that the variance of $|X^\varepsilon_t (x) - X^\varepsilon_t (y)|$ is, in the limit as $\varepsilon\downarrow 0$, of size $t^3$, independently of how close $x$ and $y$ are. We propose here a weaker form of the above property.

\begin{definition}\label{d:Richardson-dispersion}
Denote by $d$ the geodesic distance on the torus. We say that $v$ supports  weak Richardson dispersion if there is a positive number $\delta > 0$ and a set $E$ of positive measure with the property that
\[
\liminf_{\varepsilon\downarrow 0} \mathbb E[ d (X_1^\varepsilon (x), X_1^\varepsilon (y))^2] > \delta \quad\mbox{for all $(x,y) \in E\times E$.}
\]
\end{definition}

In fact the property of Definition \ref{d:Richardson-dispersion} is called {\em spontanteous stochasticity} by some authors. The next proposition clearly obstructs this weaker form. It is in fact possible to prove this obstruction directly, appealing to the Fluctuation-Dissipation relation as in \cite{Drivas} (cf. also \cite{JohanssonSorella2023}). However the following statement gives a stronger property, as it shows directly that the stochastic flows converge to deterministic (regular Lagrangian) flow as $\varepsilon \downarrow 0$.

\begin{proposition}\label{p:no-Richardson}
Assume $v$ is a (not necessarily autonomous) divergence-free vector field which satisfies the DiPerna-Lions property.  Let $X^\varepsilon_t (x)$ be as in Definition \ref{d:Richardson-dispersion} and let $X_t (x)$ be the unique regular Lagrangian flow of $v$. Then, for every positive $t$, 
\[
\lim_{\varepsilon\downarrow 0} \int_{\mathbb T^d} {\mathbb E} [ d (X^\varepsilon_t (x), X_t (x))^2]\, dx = 0\, ,
\]
where $d (y,z)$ denotes the (geodesic) distance between the points $y$ and $z$ on $\mathbb T^d$ (endowed with the flat metric). In particular $v$ cannot support Richardson dispersion.
\end{proposition}

\begin{remark}
Some authors adopt, as definition of Richardson dispersion, the statement that the variance of the one-point distribution $X^\varepsilon_t (x)$ remains strictly positive as $\varepsilon\downarrow 0$ for a nontrivial set of initial points $x$. Clearly Proposition \ref{p:no-Richardson} rules this out as well for any $v$ which satisfies the DiPerna-Lions renormalization property. 
\end{remark}

\begin{proof} Given that 
\[
{\mathbb E} [ d (X^\varepsilon_t (x), X_t^\varepsilon (y))^2]
\leq {\mathbb E} [ d (X^\varepsilon_t (x), X_t (x))^2] + {\mathbb E} [ d (X^\varepsilon_t (y), X_t (y))^2]
+ d (X_t (x), X_t (y))^2\, ,
\]
the main statement of the proposition and the usual Lusin property of measurable maps implies immediately the conclusion that $v$ cannot support Richardson dispersion.

Regard $\mathbb T^d$ as $\mathbb R^d/\mathbb Z^d$ and consider the vector field $v$ as a periodic vector field on $\mathbb R^d$. Both the solution of the stochastic ODE $X^\varepsilon_t$ and the regular Lagrangian flow $X_t$ in $\mathbb T^d$ can then be derived obviously from the corresponding objects in $\mathbb R^d$, which we denote by $Y_T^\varepsilon$ and $Y_T$. Hence, in this context it suffices to show that 
\begin{equation}\label{e:target}
\lim_{\varepsilon\downarrow 0} \int_{\Omega} {\mathbb E} [\min\{ |Y_T^\varepsilon (x) - Y_T (x)|^2, \bar C\}]
\, dx = 0
\end{equation}
for every bounded $\Omega$, every positive time $T$, and the geometric constant $\bar C = ({\rm diam}\, (\mathbb T^d))^2$. In fact if $\Omega$ contains $[0,1]^d$, then we have the inequality $d (X_T^\varepsilon (x), X_T (x))^2 \leq \min \{ |Y_T^\varepsilon (x) - Y_T (x)|^2, \bar C\}$ for all $x\in [0,1]^d \subset \Omega$ and the desired conclusion follows immediately.

\medskip

We next observe that the conclusions about the strong convergence of solutions to the parabolic equation \eqref{e:parabolic} to the unique solution of the transport equation, apply in the context of $\mathbb R^d$ as well if we assume that the initial data is in $L^2\cap L^\infty$. In fact the argument of Theorem \ref{p:no-dissipation} applies verbatim. 

\medskip

Consider now a fixed time $T$ and the parabolic equation solved {\em backward} in time on $(-\infty, T]$
\begin{equation}\label{e:backward-parabolic}
\left\{
\begin{array}{ll}
\partial_t \theta_\varepsilon + (v\cdot \nabla) \theta_\varepsilon = - \varepsilon \Delta \theta_\varepsilon\\ \\
\theta_\varepsilon (\cdot, T) = f
\end{array}
\right.
\end{equation}
together with the transport equation (solved backward in time as well)
\begin{equation}\label{e:backward-transport}
\left\{
\begin{array}{ll}
\partial_t \theta + (v\cdot \nabla) \theta = 0 \\ \\
\theta (\cdot, T) = f\, .
\end{array}
\right.
\end{equation}
If $f \in L^2\cap L^\infty$, the arguments of Theorem \ref{p:no-dissipation} imply that 
\begin{itemize}
\item[(i)] $\|\theta_\varepsilon (\cdot, 0) - \theta (\cdot, 0)\|_{L^2}$ converges to $0$ as $\varepsilon\downarrow 0$.
\end{itemize}
On the other hand we also know that 
\begin{itemize}
\item[(ii)] $\theta (x,0) = f (Y_T (x))$ for a.e. $x$ by the DiPerna-Lions theory;
\item[(iii)] $\theta_\varepsilon (x,0) = {\mathbb E} [ f (Y_T^\varepsilon (x))]$ by the Feynman-Kac formula. 
\end{itemize}
Since $\Omega$ is bounded and $Y_T$ is the regular Lagrangian flow of the bounded vector field $v$, if we choose $R$ large enough we have $Y_T (\Omega) \subset B_R$. Consider now $f := \mathbf{1}_{B_R}$ and the corresponding solutions of \eqref{e:backward-parabolic} and \eqref{e:backward-transport}. Then $\theta_\varepsilon (\cdot, 0)$ converges strongly in $L^2$ to $\theta (\cdot, 0)$, which is identically equal to $1$ on $\Omega$ because of (ii) and $Y_T (\Omega) \subset B_R$. But then $1-\theta_\varepsilon (\cdot, 0)$ converges strongly in $L^2 (\Omega)$ to $0$. Using the Feynman-Kac formula, this implies
\begin{equation}\label{e:escaping}
\lim_{\varepsilon \downarrow 0} \int_\Omega \mathbb E[1-\mathbf{1}_{B_R} (Y_T^\varepsilon (x))]\, dx = 0\, .
\end{equation}
Clearly we have 
\begin{align*}
|Y^\varepsilon_T (x) - Y_T (x)|^2
\leq & 2 |Y^\varepsilon_T (x) \mathbf{1}_{B_R} (Y^\varepsilon_T (x)) - Y_T (x)|^2
+ 2 |Y^\varepsilon_T (x)|^2 (1-\mathbf{1}_{B_R} (Y^\varepsilon_T (x)))\, , 
\end{align*}
which in turn implies 
\begin{align*}
\min \{|Y^\varepsilon_T (x) - Y_T (x)|^2, \bar C\}
\leq 2  |Y^\varepsilon_T (x) \mathbf{1}_{B_R} (Y^\varepsilon_T (x)) - Y_T (x)|^2
+ \bar C (1-\mathbf{1}_{B_R} (Y^\varepsilon_T (x))\, .
\end{align*}
In particular, from \eqref{e:escaping} we conclude that, in order to show \eqref{e:target}, it suffices to show that the integral
\begin{align*}
&\int_\Omega {\mathbb E} [|Y_T^\varepsilon (x) \mathbf{1}_{B_R} (Y_T^\varepsilon (x)) - Y_T (x) |^2]\, dx \\
= & \underbrace{\int_\Omega {\mathbb E} [|Y_T^\varepsilon (x) \mathbf{1}_{B_R} (Y_T^\varepsilon (x)) - Y_T (x)\mathbf{1}_{B_R} (Y_T (x)) |^2]\, dx}_{=: I (\varepsilon)}
\end{align*}
vanishes as $\varepsilon\downarrow 0$. 

Expanding the square and using linearity of the expectation we can write 
\begin{align*}
I (\varepsilon) &= \int_\Omega {\mathbb E} [|Y_T^\varepsilon (x)|^2 \mathbf{1}_{B_R} (Y_T^\varepsilon (x))]\, dx\\
& - 2 \int_\Omega Y_T (x) \mathbf{1}_{B_R} (Y_T (x)) \cdot {\mathbb E} [Y_T^\varepsilon (x) \mathbf{1}_{B_R} (Y_T^\varepsilon (x))]\, dx\\
&+ \int_\Omega |Y_T (x)|^2\mathbf{1}_{B_R} (Y_T (x)) \, dx\, . 
\end{align*}
Consider now the equation \eqref{e:backward-parabolic} with $f (x) = |x|^2 \mathbf{1}_{B_R} (x)$. Its solution is then given, at time $0$, exactly by the integrand of the first integral in the expression above. For $\varepsilon\downarrow 0$ we know that this converges, strongly in $L^2 (\mathbb R^d)$, to the unique solution, at time $0$, of \eqref{e:backward-transport} with the same initial data. This is then $|Y_T (x)|^2 \mathbf{1}_{B_R} (Y_T (x))$, namely the integrand of the third integral. But then, since $\Omega$ is bounded, the first integral converges, as $\varepsilon\downarrow 0$, to the third integral. 

On the other hand for the second integral we can consider the (vector-valued) initial data $f (x) = x \mathbf{1}_{B_R} (x)$ and again argue that ${\mathbb E}  [Y_T^\varepsilon (x) \mathbf{1}_{B_R} (Y_T^\varepsilon (x))]$ is the solution of the corresponding backward parabolic equation at time $0$. Again the latter converges strongly in $L^2 (\mathbb R^d)$ to the solution of the backward transport equation at time $0$, which is $Y_T (x) \mathbf{1}_{B_R} (x)$. So the second integral is also converging to the third integral. Given however the factor $-2$ in front of it, in the limit as $\varepsilon\downarrow 0$ the three expression cancel out, showing that $I (\varepsilon)\downarrow 0$. 
\end{proof}

\subsection{Yaglom's law}\label{s:Yaglom-2}
We show that, as a corollary of our results on sufficient conditions for the renormalization property, that (a possible interpretation of) the Yaglom law cannot hold under the assumptions of Theorem \ref{t:no-anomaly-2} and Theorem \ref{t:no-anomaly-3}. The Yaglom law stems from \cite{Yaglom1949} and can be interpreted or generalized in several ways. It is pertaining to classical turbulence and as such only discusses the exponent $\alpha=1/3$ and actually the behavior of a special statistical ensemble related to that exponent. In the discussion above we tied it together with the Obukhov-Corrsin relation about the balance of H\"older continuity of the vector field and that of the tracer. Our remark here addresses more directly the statistical ensemble from \cite{Yaglom1949}. It  is in a similar spirit as the works \cite{DuchonRobert2000,Eyink2003}, who provided a rigorous derivation of the Kolmogorov-$4/3$ and $4/5$ laws respectively, under the assumption of the convergence of certain spherical averages. This convergence assumption was later removed in \cite{Novack2024}. A related scaling law for the helicity was established in \cite{BoutrosTiti2025}. 

We first introduce some useful notation. For any $\xi \in \mathbb{R}^d$ we will use
\begin{equation*}
\delta f (\xi ; x, t) \coloneqq f(x + \xi, t) - f(x,t)
\end{equation*}
for the increments in the direction $\xi$.
We then define the spherical average
\begin{equation}
S(\theta,v, r) \coloneqq \int_{\mathbb{S}^{d-1}} \xi \cdot \delta v (r \xi; x,t) \lvert \delta \theta (r \xi; x,t) \rvert^2 \, d \xi,
\end{equation}
which will be our interpretation of the ensemble average in Yaglom's law (following \cite{DuchonRobert2000,Eyink2003}). Note that, when $v$ has the renormalization property, the spherical averages $S (\theta_\varepsilon, v,r)$ for the solutions $\theta_\varepsilon$ of \eqref{e:parabolic} converge to the spherical averages $S(\theta,v,r)$ for the unique bounded solution of \eqref{e:transport}.

The following theorem establishes the failure of the Yaglom law in the corresponding interpretation, under the condition that $v$ satisfies the renormalization property (and thus almost surely for the random vector fields of Theorem \ref{t:no-anomaly-2} or Theorem \ref{t:no-anomaly-3}.

\begin{theorem}\label{t:Yaglom-violated}
Assume $v$ is a (not necessarily autonomous) divergence-free vector field which satisfies the DiPerna-Lions renormalization property.
Then, if $\theta$ is the (unique) bounded solution of the transport equation \eqref{e:transport} with initial data $\theta_{in} \in L^\infty (\mathbb{T}^d)$, we have
\begin{equation}
\frac{S(\theta,v, \epsilon)}{\epsilon} \xrightarrow[]{\epsilon \rightarrow 0} 0 \, 
\end{equation}
in the sense of distributions.
\end{theorem}

\begin{proof}
Following \cite{Novack2024,BoutrosTiti2025}, we introduce a $C^\infty$ nonincreasing function $\psi : \mathbb{R} \rightarrow \mathbb{R}$ which we require to satisfy the following two conditions
\begin{equation}
\psi (s) = \begin{cases}
1 \quad \text{if } s \leq \frac{3}{4}, \\
0 \quad \text{if } s \geq \frac{5}{4}.
\end{cases} 
\end{equation}
Note that $\psi$ can be thought of as an approximation of the Heaviside function and for the rest of the argument it will not matter how it is defined in the interval $(\frac{3}{4},\frac{5}{4})$.
Then we define $\tilde{\psi}_{\eta,\epsilon} : \mathbb{R}^d \rightarrow \mathbb{R}$ be a smooth radial $C^\infty$ mollifier as follows
\begin{equation*}
\tilde{\psi}_{\eta,\epsilon} (x) \coloneqq \frac{1}{\epsilon^d} \psi \left( 1 + \eta^{-1} \bigg( \frac{\lvert x \rvert}{\epsilon} - 1 \bigg) \right),
\end{equation*}
after which we define
\begin{equation}
\psi_{\eta,\epsilon} (x) \coloneqq \frac{1}{\lVert \tilde{\psi}_{\eta,\epsilon} \rVert_{L^1}} \tilde{\psi}_{\eta,\epsilon} (x).
\end{equation}
Moreover, for any function $f$ we will use the notation $f^{\eta,\epsilon} \coloneqq f * \psi_{\eta,\epsilon}$. Mollifying the transport equation \eqref{e:transport} gives
\begin{equation} \label{mollifiedtransporteq}
\partial_t \theta^{\eta,\epsilon} + \nabla \cdot (v \theta)^{\eta,\epsilon} = 0.
\end{equation}
Taking $\theta^{\eta,\epsilon} \chi$ as a test function (for $\chi \in \mathcal{D} (\mathbb{T}^d \times (0,T))$) in the weak formulation \eqref{e:distributional}, and subtracting equation \eqref{mollifiedtransporteq} multiplied by $\theta \chi$ and integrated over $\mathbb{T}^d \times [0,T]$, gives
\begin{align*}
0 &= \int_0^\infty \int_{\mathbb T^d} \bigg[ \theta \big[\partial_t (\theta^{\eta,\epsilon} \chi) + v\cdot \nabla (\theta^{\eta,\epsilon} \chi) \big] - \theta \chi \partial_t \theta^{\eta,\epsilon} - \theta \chi \nabla \cdot (v \theta)^{\eta,\epsilon} \bigg] \, dx\, dt \\
&= \int_0^\infty \int_{\mathbb T^d} \bigg[ \theta \theta^{\eta,\epsilon} \partial_t \chi + \theta \theta^{\eta,\epsilon} v \cdot \nabla \chi + \chi \theta v \cdot \nabla \theta^{\eta,\epsilon} - \theta \chi \nabla \cdot (v \theta)^{\eta,\epsilon} \bigg] \, dx \, dt.
\end{align*}
Then we introduce the following mollified defect term
\begin{align*}
D_{\eta,\epsilon} (\theta,v) &\coloneqq \int_{\mathbb{R}^3} \nabla \psi_{\eta, \epsilon} (\xi) \cdot \delta v (\xi;x,t) \lvert \delta \theta (\xi; x,t) \rvert^2 \, d \xi \\
&= - \nabla \cdot (\lvert \theta \rvert^2 v)^{\eta,\epsilon} + v \cdot \nabla (\lvert \theta \rvert^2)^{\eta,\epsilon} + 2 \theta \nabla \cdot (\theta v)^{\eta,\epsilon} - 2 \theta v \cdot \nabla \theta^{\eta,\epsilon}.
\end{align*}
Therefore we end up with the following equation
\begin{align}
&0 = \int_0^\infty \int_{\mathbb T^d} \bigg[ \theta \theta^{\eta,\epsilon} \partial_t \chi + \theta \theta^{\eta,\epsilon} v \cdot \nabla \chi - \tfrac{1}{2} D_{\eta, \epsilon} (\theta,v) \chi - \tfrac{1}{2} \chi \nabla \cdot (\lvert \theta \rvert^2 v)^{\eta,\epsilon} + \tfrac{1}{2} \chi v \cdot \nabla (\lvert \theta \rvert^2)^{\eta,\epsilon} \bigg] \, dx \, dt \nonumber \\
&= \int_0^\infty \int_{\mathbb T^d} \bigg[ \theta \theta^{\eta,\epsilon} \partial_t \chi + \theta \theta^{\eta,\epsilon} v \cdot \nabla \chi - \tfrac{1}{2} D_{\eta, \epsilon} (\theta,v) \chi + \tfrac{1}{2} \big[(\lvert \theta \rvert^2 v)^{\eta,\epsilon} - (\lvert \theta \rvert^2)^{\eta,\epsilon} v \big] \cdot \nabla \chi  \bigg] \, dx \, dt. \label{mollifieddefecttermeq}
\end{align}
Now we first make the change of variables $\xi' = \frac{\xi}{\epsilon}$ in the formula for the mollified defect term (and relabel $\xi'$ by $\xi$)
\begin{equation}
D_{\eta,\epsilon} (\theta,v) = \frac{1}{\epsilon} \int_{\mathbb{R}^3} \nabla \psi_{\eta, 1} (\xi) \cdot \delta v (\epsilon \xi;x,t) \lvert \delta \theta (\epsilon \xi; x,t) \rvert^2 \, d \xi.
\end{equation}
By using spherical coordinates (using the radial symmetry of the mollifier), we make the change of variable $\xi = r \xi'$ (where we take $\lvert \xi' \rvert = 1$ and we again relabel $\xi'$ by $\xi$) and deduce that
\begin{align}
D_{\eta,\epsilon} (\theta,v) &= \frac{1}{\epsilon \eta} \frac{1}{\lVert \tilde{\psi}_{\eta,1} \rVert_{L^1}} \int_0^\infty \psi' \left( 1 + \eta^{-1} \big( r - 1 \big) \right) r^2 \int_{\mathbb{S}^{d-1}} \xi \cdot \delta v (\epsilon r \xi;x,t) \lvert \delta \theta (\epsilon r \xi; x,t) \rvert^2 \, d \xi \, dr \nonumber \\
&= \frac{1}{\lVert \tilde{\psi}_{\eta,1} \rVert_{L^1}} \int_0^\infty \psi' \left( 1 + \eta^{-1} \big( r - 1 \big) \right) \frac{r^3}{\epsilon \eta} \frac{S(\theta,v, \epsilon r)}{r} \, dr. \label{defecttermspherical}
\end{align}
We next show that as $\eta \rightarrow 0$, the mollified defect term converges to $- \frac{3}{4 \pi \epsilon} S(\theta,v, \epsilon)$ in the sense of distributions. To that end, we introduce another smooth radial nonnegative $C^\infty$ mollifier $\rho : \mathbb{R}^3 \rightarrow \mathbb{R}$ which we rescale as before by $\rho_\gamma (x) = \gamma^{-d} \rho (x/\gamma)$ and denote by $\theta^\delta$ the mollification of $\theta$ with $\rho_\delta$. Subtracting the term $- \frac{3}{4 \pi \epsilon} S(\theta,v, \epsilon)$ from identity \eqref{defecttermspherical}, multiplying by a test function $\chi$ and integrating in space gives
\begin{align}
\nonumber&\bigg\lvert \int_{\mathbb{T}^d} \big[ D_{\eta,\epsilon} (\theta,v) + \frac{3}{4 \pi \epsilon} S(\theta,v, \epsilon) \big] \chi \, dx \bigg\rvert \\
\nonumber&\quad = \bigg\lvert \frac{1}{\lVert \tilde{\psi}_{\eta,1} \rVert_{L^1}} \int_{\mathbb{T}^d} \chi \int_0^\infty \psi' \left( 1 + \eta^{-1} \big( r - 1 \big) \right) \frac{r^3}{\epsilon \eta} \bigg[ \frac{S(\theta,v, \epsilon r)}{r} - S(\theta,v, \epsilon) \bigg] \, dr \, dx \bigg\rvert \\
\nonumber&\quad\leq \bigg\lvert \frac{1}{\lVert \tilde{\psi}_{\eta,1} \rVert_{L^1}} \int_{\mathbb{T}^d} \chi \int_0^\infty \psi' \left( 1 + \eta^{-1} \big( r - 1 \big) \right) \frac{r^3}{\epsilon \eta} \bigg[ \frac{S(\theta^\delta,v, \epsilon r)}{r} - S(\theta^\delta,v, \epsilon) \bigg] \, dr \, dx \bigg\rvert \\
\nonumber&\quad+ \bigg\lvert \frac{1}{\lVert \tilde{\psi}_{\eta,1} \rVert_{L^1}} \int_{\mathbb{T}^d} \chi \int_0^\infty \psi' \left( 1 + \eta^{-1} \big( r - 1 \big) \right) \frac{r^3}{\epsilon \eta} \bigg[ \frac{S(\theta,v, \epsilon r)}{r} - \frac{S(\theta^\delta,v, \epsilon r)}{r} \bigg] \, dr \, dx \bigg\rvert \\
&\quad+ \bigg\lvert \frac{1}{\lVert \tilde{\psi}_{\eta,1} \rVert_{L^1}} \int_{\mathbb{T}^d} \chi \int_0^\infty \psi' \left( 1 + \eta^{-1} \big( r - 1 \big) \right) \frac{r^3}{\epsilon \eta} \bigg[ S(\theta^\delta,v, \epsilon) - S(\theta,v, \epsilon) \bigg] \, dr \, dx \bigg\rvert \,.
\end{align}
We next fix $\epsilon > 0$ and any $\gamma>0$. One can bound the terms on the fourth and fifth lines by $\lVert \theta^\delta - \theta \rVert_{L^2}$ so these terms can be made smaller than $\frac{\gamma}{3}$ by choosing $\delta$ suitable small. We fix one such $\delta$ and note that $\theta^\delta$ is continuous, so the term on the third line can be made small by choosing $\eta$ suitably small. This gives that, for $\eta$ sufficiently small, 
\[
\bigg\lvert \int_{\mathbb{T}^d} \big[ D_{\eta,\epsilon} (\theta,v) + \frac{3}{4 \pi \epsilon} S(\theta,v, \epsilon) \big] \chi \, dx \bigg\rvert < \gamma\, .
\]
Since $\gamma$ and the test function $\chi$ are arbitrary, we have shown that, for every fixed $\epsilon$, $D_{\eta,\epsilon} (\theta, v)$ converges to $-\frac{3}{4 \pi \epsilon} S(\theta,v, \epsilon)$ in the sense of distributions as $\eta\to 0$. 
Next we observe that $\tilde{\psi}_{\eta,\epsilon} \xrightarrow[]{\eta \rightarrow 0} \epsilon^{-d} \mathbf{1}_{B_\epsilon (0)}$ in $L^p (\mathbb{R}^d)$ for any $1 \leq p < \infty$. In particular $\psi_{\eta, \epsilon}  \xrightarrow[]{\eta \rightarrow 0} \psi_{0, \epsilon} := \omega_d^{-1} \epsilon^{-d} \mathbf{1}_{B_\epsilon (0)}$, where $\omega_d$ denotes the $d$-dimensional volume of the unit ball. We will extend the notation $f^{0,\epsilon}$ to mollifications of a function $f$ with $\psi^{0,\epsilon}$. Taking the limit $\eta \rightarrow 0$ in equation \eqref{mollifieddefecttermeq} (and using the distributional convergence of $D_{\eta,\epsilon}$ to $- \frac{3}{4 \pi \epsilon} S(\theta, v, \epsilon)$ shown above), we have
\begin{align}
&\int_0^\infty \int_{\mathbb T^d} \bigg[ \theta \theta^{0,\epsilon} \partial_t \chi + \theta \theta^{0,\epsilon} v \cdot \nabla \chi + \frac{1}{2} \big[(\lvert \theta \rvert^2 v)^{0,\epsilon} - (\lvert \theta \rvert^2)^{0,\epsilon} v \big] \cdot \nabla \chi  \bigg] \, dx \, dt \nonumber \\
&\qquad = - \int_0^\infty \int_{\mathbb T^d} \frac{3}{8 \pi \epsilon} S(\theta,v, \epsilon) \chi \, dx \, dt \, . \label{mollifiedenergybalance}
\end{align}
Now we observe that $(\lvert \theta \rvert^2 v)^{0,\epsilon} - (\lvert \theta \rvert^2)^{0,\epsilon} v \rightarrow 0$ in $L^1 (\mathbb{T}^d)$ as $\epsilon \rightarrow 0$. We notice that the left-hand side of this equation converges to
\begin{equation*}
\int_0^\infty \int_{\mathbb T^d} \bigg[ \lvert \theta \rvert^2 \partial_t \chi + \lvert \theta \rvert^2 v \cdot \nabla \chi \bigg] \, dx \, dt \,.
\end{equation*}
Since the vector field $v$ has the renormalization property, the above expression is equal to zero. From this equality we then deduce the following distributional limit by sending $\epsilon \rightarrow 0$ in equation \eqref{mollifiedenergybalance}
\begin{equation}
\frac{S(\theta,v, \epsilon)}{\epsilon} \xrightarrow[]{\epsilon \rightarrow 0} 0,
\end{equation}
which is what we had to show.
\end{proof}

\section{Comments on the random fields}

First of all, we observe that assumption $\mathbb P (\{v: v(x)=0\}) = 0$ in Theorem \ref{t:no-anomaly-2} and the corresponding assumption (b) in Theorems \ref{t:no-anomaly-3} and \ref{t:no-anomaly-general} certainly follow if the probability distribution of the ($\mathbb R^2$-valued) random variable $v(x)$, resp. of the ($\mathbb R^{d\times (d-1)}$-valued) random variable $D\phi (x)$, has bounded density. This is an extremely mild assumption, which is satisfied in a number of standard choices of $\mathbb P$. Consider for instance the usual construction of Gaussian random vectors on $\mathbb T^2$. More precisely let 
\[
v (x) = \sum_{ k \in \mathbb{Z}^2 \backslash \{ 0 \} } c_k v_k (x)\, ,
\]
where $\{v_k\}_{k \in \mathbb{Z}^2 \backslash \{ 0 \} }$ is an orthonormal basis of the Hilbert space
\[
H =\left\{v\in L^2 (\mathbb{T}^2): {\rm div}\, v=0 \quad \mbox{and} \quad \int_{\mathbb{T}^2} v \, dx  = 0\right\}\, .
\]
Assume the $c_k$ are independent normally distributed random coefficients with means $\mu_k$ and variances $\sigma_k>0$. Then, given any $x\in \mathbb T^2$, $v(x)$ is a normally distributed random variable with mean $\sum_{k \in \mathbb{Z}^2 \backslash \{ 0 \} } \mu_k v_k (x)$ and covariance matrix
\[
C (x) = \sum_{k \in \mathbb{Z}^2 \backslash \{ 0 \} } \sigma_k v_k (x) \otimes v_k (x)\, .
\]
It then suffices to have the lower bound $C(x) \geq C_0 {\rm Id} > 0$ in the sense of positive definite matrices to ensure that the density of the probability distribution of $v(x)$ is bounded. If we take classical bases such as Fourier or (several bases of) wavelets, the uniform lower bound is implied by the positivity of the $\sigma_k$'s. In fact the positivity of them for a handful of low frequencies is already enough. 

Of course the nondegeneracy assumption in Theorem \ref{t:no-anomaly-2} is much weaker than requiring bounded density of the pointwise distribution: it just suffices that the latter has no atom at the origin. Likewise in Theorem \ref{t:no-anomaly-3} and Theorem \ref{t:no-anomaly-general} the condition (b) can be considerably weakened. For instance it is clear from the proof that it is only needed for $M_0$'s which have rank at most $d-3$.

\medskip

Observe that the random velocity $v$ in Theorem \ref{t:no-anomaly-general} has average zero. However generalizing the statement so to include nonconstant averages is rather straightforward. We just need to make $\mathbb P$ a probability distribution on the space of $C^1$ maps $\phi: \mathbb R^d \to \mathbb R^{d-1}$ which are the sums of a $1$-periodic function and a linear map. The differential of each component $\phi_i$ is then a periodic $1$-form and induces therefore a closed $1$-form $\omega_i$ on $\mathbb T^d$. In particular the conclusion of Theorem \ref{t:no-anomaly-general} will then hold for the divergence-free vector field $w$ (which needs not be mean-free) such that $\star \sum_i w_i dx_i = \omega_1 \wedge \ldots \wedge \omega_{d-1}$. 

\bibliographystyle{acm}
\bibliography{main.bib}

\end{document}